\documentclass[reqno,11pt,a4paper]{amsart}
\usepackage{exscale,amsmath,amssymb,amsthm,color,graphicx,accents,mathrsfs,mathtools,pict2e}
\usepackage[margin=1.25in]{geometry}
\usepackage{enumitem}

\begin{document}

\newcommand{\leftrightharpoonup}{\mathrlap{\leftharpoonup}\rightharpoonup}
\newcommand{\cev}[1]{\accentset{\leftharpoonup}{#1}}
\newcommand{\vecc}[1]{\accentset{\rightharpoonup}{#1}}
\newcommand{\cvc}[1]{\accentset{\leftrightharpoonup}{#1}}
\newcommand{\scev}[1]{\accentset{\tiny\leftharpoonup}{#1}}
\newcommand{\svecc}[1]{\accentset{\tiny\rightharpoonup}{#1}}

\newcommand{\bC}{\mathbb{C}}
\newcommand{\bD}{\mathbb{D}}
\newcommand{\bE}{\mathbb{E}}
\newcommand{\bI}{\mathbb{I}}
\newcommand{\bM}{\mathbb{M}}
\newcommand{\bN}{\mathbb{N}}
\newcommand{\bP}{\mathbb{P}}
\newcommand{\bR}{\mathbb{R}}
\newcommand{\bT}{\mathbb{T}}
\newcommand{\bU}{\mathbb{U}}

\newcommand{\cB}{\mathcal{B}}
\newcommand{\cC}{\mathcal{C}}
\newcommand{\cD}{\mathcal{D}}
\newcommand{\cE}{\mathcal{E}}
\newcommand{\cF}{\mathcal{F}}
\newcommand{\cG}{\mathcal{G}}
\newcommand{\cH}{\mathcal{H}}
\newcommand{\cI}{\mathcal{I}}
\newcommand{\cK}{\mathcal{K}}
\newcommand{\cL}{\mathcal{L}}
\newcommand{\cM}{\mathcal{M}}
\newcommand{\cN}{\mathcal{N}}
\newcommand{\cP}{\mathcal{P}}
\newcommand{\cT}{\mathcal{T}}
\newcommand{\cR}{\mathcal{R}}
\newcommand{\cS}{\mathcal{S}}
\newcommand{\cU}{\mathcal{U}}

\newcommand{\sD}{\mathscr{D}}

\newcommand{\ff}{\mathbf{f}}
\newcommand{\ft}{\mathbf{t}}
\newcommand{\fs}{\mathbf{s}}
\newcommand{\fE}{\mathbf{E}}
\newcommand{\fF}{\mathbf{F}}
\newcommand{\fN}{\mathbf{N}}
\newcommand{\fn}{\mathbf{n}}
\newcommand{\fX}{\mathbf{X}}
\newcommand{\fP}{\mathbf{P}}
\newcommand{\fT}{\mathbf{T}}
\newcommand{\bgamma}{\gamma}

\newcommand{\loc}{\ell}

\newtheorem{lemma}{Lemma}
\newtheorem{proposition}[lemma]{Proposition}
\newtheorem{corollary}[lemma]{Corollary}
\newtheorem{problem}{Problem}
\newtheorem{theorem}[lemma]{Theorem}
\newtheorem{aside}{Aside}
\newtheorem{hypothesis}{Hypothesis}
\newtheorem{definition}[lemma]{Definition}

\newenvironment{pfofextconv1}{\begin{trivlist}\item[] \textbf{Proof of Proposition \ref{extconv1}.}}
                     {\hspace*{\fill} $\square$\end{trivlist}}

\newcommand{\ed}{\mbox{$ \ \stackrel{d}{=}$ }}
\newcommand{\te}{\rightarrow}
\newcommand{\giv}{\,|\,}
\newcommand{\convd}{\overset{d}{\underset{{n\rightarrow \infty}}{\longrightarrow}}}
\newcommand{\convdk}{\overset{d}{\underset{{k\rightarrow \infty}}{\longrightarrow}}}
\newcommand{\convask}{\overset{a.s.}{\underset{{k\rightarrow \infty}}{\longrightarrow}}}

\newcommand{\convv}{\overset{v}{\underset{{n\rightarrow \infty}}{\longrightarrow}}}
\newcommand{\convvz}{\overset{v}{\underset{{z\downarrow 0}}{\longrightarrow}}}
\newcommand{\convdz}{\overset{d}{\underset{{z\downarrow 0}}{\longrightarrow}}}

\newcommand{\besq}{{\tt BESQ}}
\newcommand{\ekp}{{\tt EKP}}
\newcommand{\pd}{{\tt PD}}
\newcommand{\skewer}{\ensuremath{\normalfont\textsc{skewer}}}
\newcommand{\skewerbar}{\ensuremath{\overline{\normalfont\textsc{skewer}}}}
\newcommand{\clade}{\ensuremath{\normalfont\textsc{clade}}}
\newcommand{\cutoff}{\ensuremath{\normalfont\textsc{cutoff}}}

\title[Squared Bessel processes in Brownian local times]{Squared Bessel processes of positive and negative dimension embedded in Brownian local times\vspace{-0.15cm}}

\author{J\MakeLowercase{\sc im} P\MakeLowercase{\sc itman and}  M\MakeLowercase{\sc atthias} W\MakeLowercase{\sc inkel}\vspace{-0.15cm}}

\address{\hspace{-0.42cm}J.~Pitman\\ Statistics Department\\ Evans Hall \#3860\\ University of California\\
Berkeley CA 94720\\ USA\\ Email: pitman@berkeley.edu}             

\address{\hspace{-0.42cm}M.~Winkel\\ Department of Statistics\\ University of Oxford\\ 24--29 St Giles'\\ Oxford OX1 3LB\\ UK\\ Email: winkel@stats.ox.ac.uk}

\keywords{Brownian motion, local times, excursions, squared Bessel processes}

\subjclass[2010]{60J80}

\date{\today}

\thanks{This research has been partially supported by NSF grant DMS-1444084 and the Astor Travel Fund of the University of Oxford}

%\author{Noah Forman \and Soumik Pal \and Douglas Rizzolo \and Matthias Winkel}             

\begin{abstract}\noindent The Ray--Knight theorems show that the local time processes of various path fragments derived from a one-dimensional
Brownian motion $B$ are squared Bessel processes of dimensions $0$, $2$, and $4$.
It is also known that for various singular perturbations $X= |B| + \mu \loc$ of a reflecting Brownian motion $|B|$ by a multiple $\mu$ of its local time process $\loc$ at $0$, 
corresponding local time processes of $X$ are squared Bessel with other real dimension parameters, both positive and negative.
Here, we embed squared Bessel processes of all real dimensions directly in the local time process of $B$. This is done by decomposing the path of $B$
into its excursions above and below a family of continuous random levels determined by the Harrison--Shepp construction of skew Brownian motion as the strong solution
of an SDE driven by $B$.  This embedding connects to Brownian local times a framework of point processes of squared Bessel excursions of negative dimension and 
associated stable processes, recently introduced by Forman, Pal, Rizzolo and Winkel to set up interval partition evolutions that arise in their approach to 
the Aldous diffusion on a space of continuum trees. 
%%%{\hfill\tt work in progress}
\end{abstract}

\maketitle

\vspace{-0.6cm}

\section{Introduction and statement of main results}

Squared Bessel processes are a family of one-dimensional diffusions on $[0,\infty)$, defined by 
continuous solutions $Y = (Y(x), 0 \le x \le \zeta)$ of the stochastic differential equation
\begin{equation}
\label{besqdef}
  dY(x)=\delta\,dx+2\sqrt{Y(x)}dB(x),\qquad Y(0)=y\ge 0, \qquad 0 < x < \zeta
\end{equation}
where $\delta$ is a real parameter, $B=(B(x),x\ge 0)$ is standard Brownian motion,
and $\zeta$ is the {\em lifetime} of $Y$, defined by 
\begin{equation}
\label{def:sdecases}
  \zeta:=\begin{cases}
    \infty & \text{if $ \delta > 0 $ }\\
    T_0:= \inf \{x\ge 0\colon Y(x) = 0 \}  & \text{if $\delta \le 0$}.
  \end{cases}
\end{equation}
%It is known \cite[Chapter XI]{RevuzYor} that for $\delta > 0$, this SDE has a unique strong solution for $0 \le x < \zeta = \infty$.
%It is also known \cite{GoinYor03} that for $\delta \le 0$ and $y >0$ this SDE has a unique strong solution with $0 < \zeta < \infty$ and 
%$Y_\zeta = 0$ almost surely.  For $\delta \le 0$ we take the boundary state $0$ to be absorbing, setting $Y(x) = 0$ for $x \ge \zeta$ if
%the process is started in state $y >0$, and $Y(x) \equiv 0$ if $Y(0) = 0$.
It is known \cite[Chapter XI]{RevuzYor}, \cite{GoinYor03}  that this SDE has a unique strong solution, with
$\zeta < \infty$ and $Y(\zeta) = 0$ almost surely if $\delta \le 0$, when we make the boundary state $0$ absorbing by setting $Y(x) = 0$ for $x \ge \zeta$.
The distribution on the path space $C[0,\infty)$ of the process $Y$ so defined, for each starting state $y \ge 0$ and
each real $\delta$, is denoted ${\tt BESQ}_y(\delta)$. For each real $\delta$, the collection of laws
$({\tt BESQ}_y(\delta), y \ge 0)$ defines a Markovian diffusion process on $[0,\infty)$, the {\em squared Bessel process of dimension $\delta$},
denoted ${\tt BESQ}(\delta)$.
 
Several other constructions and interpretations of ${\tt BESQ}(\delta)$ are known. % and reviewed briefly in Section \ref{sec:back}.
In particular, 
\begin{itemize}[leftmargin=.7cm]
\item ${\tt BESQ}(\delta)$ for $\delta = 1,2, \ldots$ is the squared norm of standard Brownian motion in $\bR^d$;
\item ${\tt BESQ}(\delta)$ may be understood for all real $\delta$ as a continuous-state branching process,  with an immigration rate 
$\delta$ if $\delta >0$, emigration rate $|\delta|$ if $\delta <0$, and lifetime $\zeta$ at which the population dies out.
\end{itemize}
The case of immigration has been well-studied \cite{KawWat71,ShiWat73,RogWil2,RevuzYor,Lambert2002,Li2006}. 
The literature on the case of emigration is rather sparse \cite{Pal13,Paper3}, but scaling limit results for discrete branching processes with emigration \cite{Vatutin1977,VatZub1993} with
${\tt BESQ}(\delta)$ limits for $\delta <0$ can be obtained from 
\cite{Alexander2011,BerKor2016,Pal13}. 
While dimension $\delta = 0$ is critical for whether the ${\tt BESQ}(\delta)$ process has finite or infinite lifetime,
dimension $\delta = 2$ is well known to be critical in another respect: for a ${\tt BESQ}_v(\delta)$ process $Y$ with hitting times $T_y:= \inf \{x\ge 0\colon Y(x) = y \}$, 
\begin{itemize}[leftmargin=.7cm]
\item for $\delta\!>\!2$ the process is upwardly transient, with 
$\bP_v(T_0\!<\!\infty)\!=\!0$ and $\bP_v(T_w\!<\!\infty)\!=\!1$ for $0< v < w$, while
\item for $\delta <2$ the process is either recurrent if $0 < \delta < 2$, or downwardly transient 
%%%with absortion at $0$ 
if $\delta \le 0$, with $\bP_v(T_a < \infty) = 1$ for all $0 \le a < v$ in either case.
\end{itemize}
A remarkable duality between ${\tt BESQ}(\delta)$ processes of dimensions $\delta = 2 \pm 2 \alpha $
was pointed out in \cite[Theorem (3.3) and Remark (4.2)(ii)]{MR620995} %%%AUTHOR = {Pitman, Jim and Yor, Marc}, TITLE = {Bessel processes and infinitely divisible laws}, 
and \cite[Section 3]{PitmYor82}:
\begin{itemize}[leftmargin=.7cm]
\item for each real $\alpha \ge 0$ and $0 < u < v$, the conditional distribution of a ${\tt BESQ}_v( 2 + 2 \alpha )$ process up to time $T_u$, given the event $(T_u < \infty)$ which has
probability $(u/v)^{\alpha}$, equals the unconditional distribution of a ${\tt BESQ}_v( 2 - 2 \alpha )$ process up to time $T_u$.
\end{itemize}
There is a similar description of ${\tt BESQ}_u( 2 + 2 \alpha )$ up to $T_v$ for $0 < u < v$ as ${\tt BESQ}_u( 2 - 2 \alpha )$ up to $T_v$ given $T_v < T_0$.
This duality relation between dimensions $2 \pm 2 \alpha $ is best known for $\alpha \in (0,1)$. Then it relates
the recurrent dimensions $2 - 2 \alpha  \in (0,2)$, for which the inverse local time process of ${\tt BESQ}_0(2 - 2 \alpha)$  at $0$ is a stable subordinator of index 
$\alpha$, to the transient dimensions $2 + 2 \alpha \in (2,4)$. For $\alpha = 1/2$
this is the well known relation between Brownian motion on $[0,\infty)$ with either reflection or absorption at $0$, and the three-dimensional Bessel process, 
expressed here in terms of squared Bessel processes.
But as emphasized in \cite[Example (3.5)]{PitmYor82}, the duality relation between dimensions $2 \pm 2\alpha$ holds also for $\alpha \ge 1$, when it relates the downwardly 
transient ${\tt BESQ}(- \delta)$  process for $- \delta = 2 - 2 \alpha \le 0$ to the upwardly transient ${\tt BESQ}(4+ \delta)$ process.

It was shown by Shiga and Watanabe \cite{ShiWat73} that the distribution of ${\tt BESQ}_y(\delta)$ for all real $y \ge 0$ and $\delta \ge 0$ is
uniquely determined by the prescription that ${\tt BESQ}_y(1)$ is the distribution of $(\sqrt{y} + B)^2$, and the following additivity property: for $y,y^\prime\ge 0$ and $\delta,\delta^\prime\ge 0$,
and two independent processes $Y$ and $Y'$,
\begin{equation}\label{fulladd}
\mbox{if $Y$ is  a ${\tt BESQ}_y(\delta)$ and $Y^\prime$ is a ${\tt BESQ}_{y^\prime}(\delta^\prime)$ then $ Y+Y^\prime$ is a ${\tt BESQ}_{y+y^\prime}(\delta+\delta^\prime)$}.
\end{equation}
The distribution of ${\tt BESQ}_y(-\delta)$ for all $y > 0$ and $\delta >0$ is determined in turn by the duality between dimensions $-\delta$ and $4 + \delta$.
Pitman and Yor \cite{PitmYor82} used the additivity property to construct a ${\tt BESQ}_y(\delta)$ process $Y_y^{(\delta)}$  for $y, \delta \ge 0$
as a sum of points in a $C[0,\infty)$-valued Poisson point process, whose intensity measure involves the local time profile induced by It\^o's law of Brownian excursions.
The $C[0,\infty)$-valued process $(Y_y^{(\delta)}\!,\,y \!\ge\! 0,\delta \!\ge\! 0)$ then has stationary independent increments in both $y \!\ge\! 0$ and $\delta \!\ge\! 0$.
This construction, and the duality between dimensions $0$ and $4$, explained the multiple appearances of ${\tt BESQ}(\delta)$ processes and their bridges for $\delta = 0,2$ and $4$ in the Ray--Knight descriptions of Brownian local time 
processes.

This model of Brownian local times and ${\tt BESQ}$ processes, driven by a Poisson point process of local time pulses from Brownian excursions, led to a number of further developments.
In particular, as recalled later in Lemmas \ref{lmplus} and \ref{lmminus}, if a reflecting Brownian motion $|B|$ is perturbed by adding 
a multiple $\mu$ of its local time process $\loc$ at $0$, to form $X:= |B| + \mu \loc$, where $\mu$ might be of either sign,
then the resulting {\em perturbed Brownian motion} $X$ has a local time process from which it is possible, by varying $\mu$, and sampling 
%%%the local time process of $X$ 
at suitable random times, to construct ${\tt BESQ}(\delta)$ processes for all real $\delta$.  
The more recent notion of a {\em Poisson loop soup} 
\cite{MR2815763} %%AUTHOR = {Le Jan, Yves},
greatly generalizes this construction of local time fields from one-dimensional Brownian motion to one-dimensional diffusions \cite{lupu-diffusion-loops} 
and much more general Markov processes.

Despite these constructions of ${\tt BESQ}(\delta)$ for all real $\delta$ in  the local time processes of perturbed Brownian motions,
and the general importance of additivity properties in the construction of local time fields \cite{MR2250510} %% AUTHOR = {Marcus, Michael B. and Rosen, Jay}, TITLE = {Markov processes, {G}aussian processes, and local times},
\cite{MR2815763}, %%AUTHOR = {Le Jan, Yves},
it is known \cite[Exercise XI.(1.33)]{RevuzYor} and \cite[top of p.332]{GoinYor03} that the additivity property \eqref{fulladd}
of ${\tt BESQ}$ processes 
fails
without the assumption that both $\delta \ge 0$ and $\delta' \ge 0$.
Our starting point here is a weaker form of additivity of ${\tt BESQ}$ processes, involving both positive and negative 
dimensions:
\begin{proposition}
\label{prop:add} For arbitrary real $\delta,\delta^\prime$ and $y,y^\prime\ge 0$, let $Y$, $Y^\prime$  and $Y_1$ be three independent processes, with 
\begin{itemize}[leftmargin=.7cm]
\item $Y$ a ${\tt BESQ}_y(\delta)$ with lifetime $\zeta$;
\item $Y'$ a ${\tt BESQ}_{y'}(\delta')$ with lifetime $\zeta'$;
\item $Y_1$ a ${\tt BESQ}_{1}(\delta + \delta')$.
\end{itemize}
Let $T$ be a stopping time relative to the filtration generated by the pair of processes $(Y,Y')$, with $T \le \zeta \wedge \zeta'$, and let
$Z$ be the process\vspace{-0.1cm}
\begin{equation}
\label{def:cases}
  Z(x):=\begin{cases}
    Y(x) + Y'(x), & \text{if $0 \le x \le T$},\\
    Z(T) Y_1( (x - T)/Z(T) ), & \text{if $T < x < \infty$}.
  \end{cases}\vspace{-0.1cm}
\end{equation}
Then $Z$ is a ${\tt BESQ}_{y + y'}(\delta + \delta')$.
\end{proposition}

By the scaling property of squared Bessel processes, for each fixed $z >0$ the scaled process $(z Y_1(w/z), w \ge 0)$ is a ${\tt BESQ}_{z}(\delta + \delta')$.
So \eqref{def:cases} sets $Z := Y + Y'$ on $[0,T]$, and makes $Z$ evolve as a ${\tt BESQ}(\delta + \delta')$ on $[T,\infty)$.
This proposition and its proof are a straightforward generalization of the case with $\delta = -1$, $\delta' = 0$ and 
$T = \zeta \wedge \zeta'$, which was established as \cite[Lemma 25]{Paper3}. 
The proof is by consideration of the SDE \eqref{besqdef}, as in the proof of the additivity property (\ref{fulladd}) in \cite[Theorem XI (1.2)]{RevuzYor}.
If both $\delta, \delta' \ge 0$, the conclusion Proposition \ref{prop:add} holds even without the assumption $T \le \zeta \wedge \zeta'$,
by combining the simpler additivity property (\ref{fulladd}) with the strong Markov property of ${\tt BESQ}(\delta + \delta')$.  
We are particularly interested in the instance of Proposition \ref{prop:add} with $y = 0$, $y' = v$,
 $\delta' = - \delta<0$ and $T = \zeta'$, which may be paraphrased as follows:

\begin{corollary}
\label{propadd} Let $\delta>0$ and $v\ge 0$. Let $Y^\prime:=
%%%Y_v^\prime=
(Y_v^\prime(x),x\ge 0)$ be ${\tt BESQ}_v(-\delta)$ 
absorbed at $\zeta^\prime(v):=\inf\{x\ge 0\colon Y_v^\prime(x)=0\}$. Conditionally given $Y^\prime$ with $\zeta^\prime(v)=a$, let 
$Y:=
%%%Y^{(v)}_0=(
(Y_{0,v}^{(\delta)}(x),x\ge 0)$ be a 
%%%n independent 
time-inhomogeneous 
  Markov process that is ${\tt BESQ}_0(\delta)$ on the time interval $[0,a]$ and then continues on $[a,\infty)$ as ${\tt BESQ}(0)$.\ Then $Y+Y^\prime$ is a ${\tt BESQ}_v(0)$. 
\end{corollary}
The subtlety here is 
that we create dependence between $Y$ and $Y^\prime$ by specifying that $Y$ only follows 
%the 
${\tt BESQ}_0(\delta)$ 
%process 
independently
of $Y'$ until time
$\zeta^\prime(v)$, when $Y^\prime$ hits zero,
and then $Y$ continues as needed for the additivity to hold. 
%We provide further generalisations and 
%a proof of Proposition \ref{propadd} in Section \ref{sec:add}. 
In \cite{Paper1}, the authors encountered the case $\delta=1$ of Corollary \ref{propadd} in a more elaborate  context which we review in Section %\ref{sec:FPRW}. 
\ref{sec:lit}.

Let $L=(L(x,t),x\in\bR, t\ge 0)$ be the jointly continuous space-time local time process of Brownian motion $B=(B(t),t\ge 0)$ and let 
$\tau(v)=\inf\{t\ge 0\colon L(0,t)>v\}$ be the inverse local time of $B$ at $0$. 
According to one of the Ray--Knight theorems, 
%%%${\tt BESQ}_v(0)$ can be found in the local time process of Brownian motion. Specifically, denote by 
the process $(L(x,\tau(v)),x\ge 0)$ is a ${\tt BESQ}_v(0)$. This raises the following question:
\medskip
\begin{center}
Can we find the pair $(Y,Y^\prime)$ of Corollary \ref{propadd} embedded in the local times of $B$?
\end{center}
\medskip
%The same question can be asked for $Y\sim{\tt BESQ}_v(-\delta)$, paired with $Z$, now a ${\tt BESQ}_0(\delta)$ during the lifetime of $Y$ and continuing 
%as ${\tt BESQ}(0)$, for any $\delta>0$. 
The following theorem provides a positive answer to this question.
See Figure \ref{splitfig} for an illustration of the embedding.
\begin{theorem}\label{emimm} For each $\delta>0$, there is an increasing family of stopping times $S_\delta(x), x \ge 0$ %{\tt unique, subject to some right-continuity in $x$?} 
  such that the following two families of random variables are independent:\vspace{-0.1cm}
  \begin{itemize}[leftmargin=.7cm]
    \item $Y_0^{(\delta)}:=(L(x,S_\delta(x)),x\ge 0) \ed {\tt BESQ}_0(\delta)$;
    \item $Y_v^\prime:=(L(x,\tau(v))-L(x,S_\delta(x)\wedge\tau(v)),x\ge 0)\ed {\tt BESQ}_v(-\delta)$ for all $v\ge 0$.
  \end{itemize}
  For each $v\ge 0$, the random level $\zeta^\prime(v):=\inf\{x\ge 0\colon S_\delta(x)>\tau(v)\}$ is almost surely finite,
 and coincides with the 
  absorption time of $Y_v^\prime$. Conditionally given $\zeta^\prime(v)=a$, 
  \begin{itemize}[leftmargin=.7cm]
	\item 
the process 
$Y_{0,v}^{(\delta)}:=(L(x,S_\delta(x)\wedge\tau(v)),x\ge 0)$ is independent of $Y_v^\prime$ and a time-inhomogeneous Markov process that is 
       ${\tt BESQ}_0(\delta)$ on the time interval $[0,a]$ and then continues as ${\tt BESQ}(0)$.%\pagebreak
  \end{itemize} 
\end{theorem} 
\begin{figure}
  \begin{picture}(415,190)
          \put(30,0){\includegraphics[height=190pt,width=355pt]{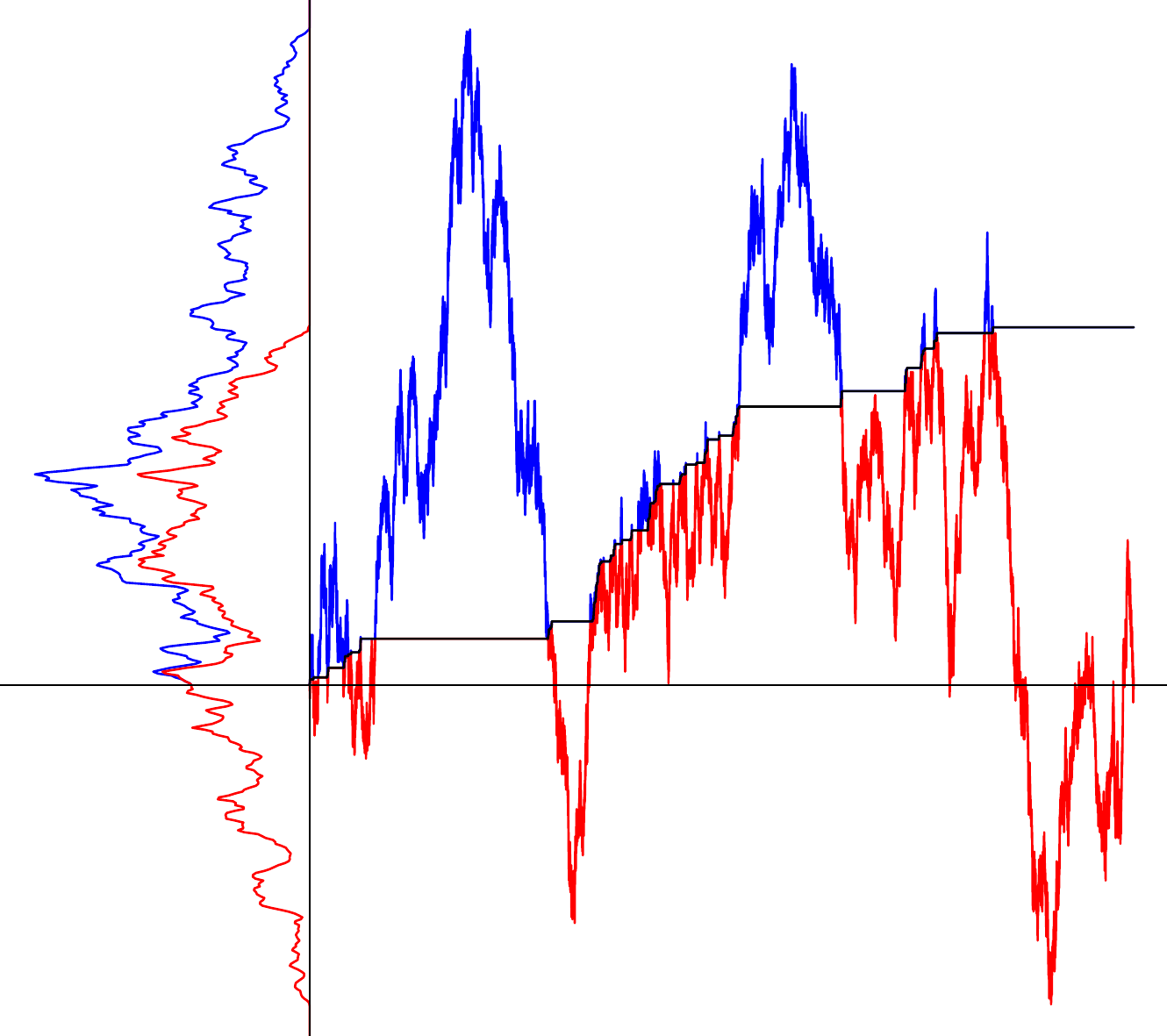}}
          \put(85,130){\line(1,0){40}}
          \multiput(125,130)(10,0){21}{\line(1,0){5}}
          \put(65,130){$\scriptstyle\zeta^\prime(v)$}
          \put(365,100){$\scriptstyle B(t)$}
          \put(380,66){$\scriptstyle t$}
          \put(128,185){$\scriptstyle x$}
%          \put(119,185){\line(1,0){12}}
%          \put(100,183){$\scriptstyle\zeta(v)$}
%          \put(130,145){$\scriptstyle\zeta^\prime(v)$}
          \put(85,85){$\scriptstyle {\tt BESQ}_v(-\delta)$}
          \put(35,115){$\scriptstyle {\tt BESQ}_0(\delta)$}
          \put(50,110){\vector(3,-2){20}}
          \put(65,170){$\scriptstyle {\tt BESQ}(0)$}
          \put(90,165){\vector(4,-1){25}}
          \put(45,45){$\scriptstyle {\tt BESQ}_v(0)$}
          \put(80,50){\vector(5,1){30}}
          \put(376,58){\line(0,1){12}}
          \put(372,49){$\scriptstyle\tau(v)$}
  \end{picture}
  \caption{Simulation of Brownian motion $(B(t),0\!\le\! t\!\le\!\tau(v))$ split along a suitable increasing path determined by $\delta >0$. 
The red excursions below the path have total local time process ${\tt BESQ}_v(-\delta)$ on $[0,\infty)$.  The blue excursions above the path have total local time process on $[0,\infty)$ which is ${\tt BESQ}_0(\delta)$ up to level $\zeta^\prime(v)$, then continue as ${\tt BESQ}(0)$ 
above level $\zeta^\prime(v)$.\label{splitfig}}
\end{figure}
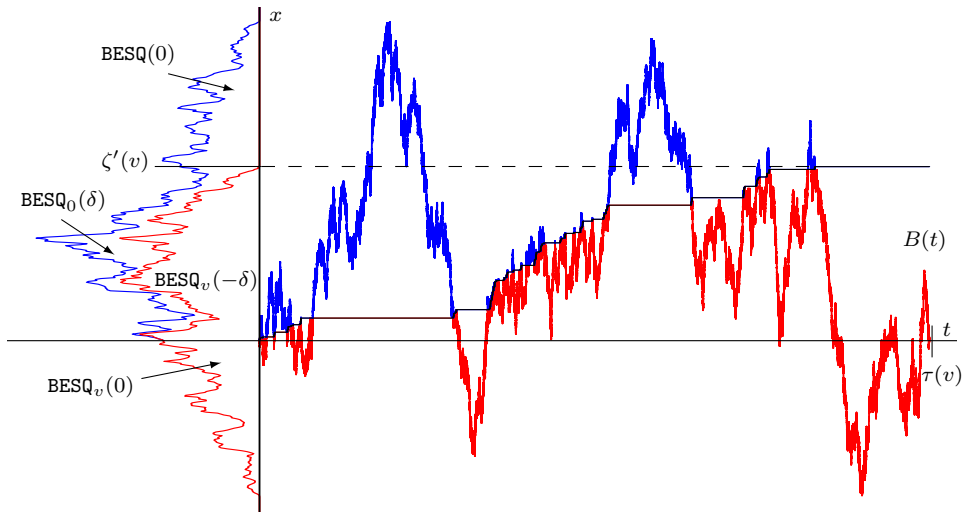

Thinking of ${\tt BESQ}(\delta)$ as a branching processes with immigration or emigration, according to the sign of $\delta$, 
Theorem \ref{emimm} provides a frontier $S_\delta$ varying with $x$,
 across which the emigration of ${\tt BESQ}(-\delta)$ is the immigration of ${\tt BESQ}(\delta)$.

\begin{corollary}\label{cor10} In the setting of Theorem \ref{emimm}, the $C[0,\infty)$-valued process $(Y_0^{(\delta)},\delta\ge 0)$, with $Y_0^{(0)}\equiv 0$,
  has stationary and independent increments in $\delta\ge 0$. 
%Let $\delta_1\!>\!\delta_2\!>\!0$. Then $(L(x,S_{\delta_1}(x))\!-\!L(x,S_{\delta_2}(x)),x\!\ge\! 0)\!\ed\!{\tt BESQ}_0(\delta_1\!-\!\delta_2)$.
\end{corollary} 

The weaker form of additivity in Theorem \ref{emimm} raises further questions. Here are some:\vspace{0.2cm}
\begin{enumerate}[leftmargin=.7cm]
  \item[1.] Is the (right-continuous increasing) process $(S_\delta(x),x\ge 0)$ of stopping times uniquely identified by the distribution of $Y_0^{(\delta)}\!$ specified in the first bullet point of Theorem \ref{emimm}?\vspace{0.2cm}
  \item[2.] In Corollary \ref{propadd}, what is the conditional distribution of $Y^\prime$ or of $\zeta^\prime(v)$ given $Y+Y^\prime$?\vspace{0.2cm}
  \item[3.] Suppose a non-negative process $Y'$ absorbed at 0 at time $\zeta'$ is such that
    $Y + Y'$ is ${\tt BESQ}_v(0)$ for $Y$ conditionally given $Y'$ as in Corollary \ref{propadd}.
    Is $Y'$ a ${\tt BESQ}_v(- \delta)$ process?\vspace{0.2cm}
%is the law of $Y^\prime$, 
%    on the space of paths from $v$ to absorption in 0 at a time that we denote by $\zeta^\prime$, 
%    uniquely identified by the definition of $Z$ from $Y^\prime$, $Y$ and $Y_1$ and the distributions of $(Y,Y_1)$ and $Z$? 
%  When $\delta^\prime\ge 0$, there is no need for $Y_1$ and the claim is an easy consequence of (\ref{fulladd}).
\end{enumerate}
%\vspace{7.5cm}

The rest of this article is organized as follows: %Section \ref{sec:back} provides some more background on Bessel and squared Bessel processes. 
Section \ref{sec:proof} presents the proofs of Theorem \ref{emimm} and Corollary \ref{cor10}. In Section \ref{sec:checks}, we explore the implications of Proposition \ref{prop:add} 
%%%and Corollary \ref{propadd} 
by checking the laws of some marginals and functionals. We conclude in Section \ref{sec:lit} by pointing out some related developments.

\section{Proofs of Theorem \ref{emimm} and Corollary \ref{cor10}}
\label{sec:proof}

Our proof of Theorem \ref{emimm} exploits the following known variants of the Ray--Knight theorems for perturbed Brownian motions
$R_\mu ^\pm := |B| \pm \mu \ell$, where $\ell$ is the local time of $B$ at $0$ normalized so that $|B| - \ell \ed B$,
and we assume $\mu >0$.

\begin{lemma}[Th\'eor\`eme 2 of Le Gall and Yor \cite{LGY1986}, Theorems 3.3-3.4 of \cite{YorAspects1}]%Lemma 2.1 of Perman \cite{Perman1996}]
  \label{lmplus} 
 %% Let $\mu>0$ and $a\in[0,\infty]$. Set $R_\mu^+(t):=|B(t)|+\mu L(0,t)$. Then 
The space-time local time process 
  $(L^+_\mu(x,t),x\in\bR,t\ge 0)$ of $R_\mu^+$ is such that for each $a\in[0,\infty]$
  $$(L_\mu^+(x,\tau(a/\mu)),x\ge 0) \mbox{ is }{\tt BESQ}_0(2/\mu)\mbox{ on }[0,a]\mbox{ continued on $[a,\infty)$ as }{\tt BESQ}(0).$$
\end{lemma}%\pagebreak
\begin{lemma}[Theorem 3.3 of Carmona, Petit and Yor \cite{CPY1994}]
\label{lmminus} 
%Let $\mu>0$ and $v\ge 0$. Set $R_\mu^-(t):=|B(t)|-\mu L(0,t)$. Then 
For each fixed $v\ge 0$ the  local time process 
  $(L^-_\mu(x,t),x\!\in\!\bR,t\!\ge\! 0)$ of $R_\mu^-$ evaluated at
$\tau_\mu^-(v)\!:=\!\inf\{t\!\ge\! 0\colon L_\mu^-(0,t)\!>\!v\}$ yields independent 
processes
\vspace{-0.1cm}
  $$(L_\mu^-(-x,\tau_\mu^-(v)),x\ge 0)\ed {\tt BESQ}_v(2-2/\mu)\quad\mbox{and}\quad(L_\mu^-(x,\tau_\mu^-(v)),x\ge 0)\ed{\tt BESQ}_v(0).$$
\end{lemma}
%%See also \cite{NRW1987,YorAspects1,CPY1994,Perman1996,Pitman1996,PerWer1997,DMY1997,Warren1997,CPY1998,Doney1998,DWY1998,Pitman1999,DuLG,Lambert2002,CPY2004,lupu-diffusion-loops} for numerous extensions and ramifications of these results.

Now let $\bgamma\in[-1,1]$. Consider the excursions away from level 0 of reflected Brownian motion. Independently multiply each excursion by $-1$ with probability 
%%%$\frac{1}{2}(1+\bgamma)$. 
$\frac{1}{2}(1-\bgamma)$. 
The resulting 
process $X_\bgamma=(X_\bgamma(t),t\ge 0)$ is known as \em skew Brownian motion\em. See \cite{Lejay} for a recent survey of constructions of this process, including the construction of $X_\bgamma$ by
Harrison and Shepp \cite{HarrShep81} as
the unique strong solution to the equation
\begin{equation}
\label{skewbm}
X_\bgamma(t)=B(t)-\bgamma \ell_\bgamma(t),\qquad t\ge 0,
%%%X_\bgamma(t)=B(t)+\bgamma \ell_\bgamma(t),\qquad t\ge 0,
\end{equation}
where $B$ is Brownian motion and $\ell_\bgamma$ is the local time process at 0 of $X_\bgamma$, that is 
%%$$\ell_\bgamma(t)=\lim_{h\downarrow 0}\frac{1}{2h}\int_0^t1_{[-h,h]}(X_\bgamma(s))ds.\vspace{-0.1cm}$$
$$\ell_\bgamma(t)=\lim_{h\downarrow 0}\frac{1}{2h}\int_0^t 1 \{-h < X_\bgamma(s) < h \} ds. $$
where the limit exists simultaneously for all $t\ge 0$ almost surely.
This choice of local time at 0 is defined so that $\ell_\bgamma(\cdot)\overset{d}{=}\loc$ for all $\bgamma$, where $\loc:= \ell_0=L(0,\,\cdot\,)$  is the usual
local time of $|B|$ at $0$, normalised as in Lemmas \ref{lmplus} and \ref{lmminus}, so that $|B|-\loc\overset{d}{=}B$. 

\begin{lemma}\label{propskew} For $\bgamma\!\in\!(0,1)$, let $X_\bgamma$ be the skew Brownian motion driven by $B$ as in {\em (\ref{skewbm})}. 
Let $I^+\!:=\! (0,\infty)$ and $I^-\!:=\! (-\infty,0)$, and
%%%For $\pm \in \{+, -\}$  let
consider time changes 
$\kappa^\pm_\bgamma(s)\!=\!\inf\left\{t\!\ge\! 0\colon A_\bgamma^\pm(t)\!>\!s\right\}$
where
 $A_\bgamma^\pm(t)\!:=\!\int_0^t 1 \{X_\bgamma(r) \in I^\pm \}dr$.
Then
%  $A_\bgamma^-(t)\!:=\!\int_0^t 1_{\{X_\bgamma(r)\le 0\}}dr$. Then 
%consider time changes $\kappa^+_\bgamma(s)\!=\!\inf\left\{t\!\ge\! 0\colon A_\bgamma^+(t)\!>\!s\right\}$, where 
%  $A_\bgamma^+(t)\!:=\!\int_0^t 1_{\{X_\bgamma(r)\ge 0\}}dr$ and $\kappa^-_\bgamma(s)\!=\!\inf\left\{t\!\ge\! 0\colon A_\bgamma^-(t)\!>\!s\right\}$, where 
%  $A_\bgamma^-(t)\!:=\!\int_0^t 1_{\{X_\bgamma(r)\le 0\}}dr$. Then 
  \begin{itemize}[leftmargin=.7cm]
    %\item $W_\bgamma^+:=B\circ\kappa_\bgamma^+\overset{d}{=}|B|+\mu_\bgamma^+\ell$, where $\mu_\bgamma^+=2\bgamma/(1-\bgamma)$,
    %\item $W_\bgamma^-:=-B\circ\kappa_\bgamma^-\overset{d}{=}|B|-\mu_\bgamma^-\ell$, where $\mu_\bgamma^-=2\bgamma/(1+\bgamma)$,\vspace{0.1cm}
    \item $W_\bgamma^\pm := \pm B\circ\kappa_\bgamma^\pm \overset{d}{=}|B| \pm \mu_\bgamma^\pm \loc $, where $\mu_\bgamma^\pm=2\bgamma/(1- \pm \bgamma) > 0$,\vspace{0.1cm}
    \item $W_\bgamma^+$ and $W_\bgamma^-$ are independent.\pagebreak[2]
  \end{itemize}
\end{lemma}
\begin{proof} Denote by $n_{\rm ex}$ the excursion intensity measure of reflecting Brownian motion relative to increments of $\loc$.  See e.g. 
  \cite[Chapter XII]{RevuzYor}. 
Relative to increments of the local time process $\ell_\bgamma$ of $X_\bgamma$, with $\ell_\bgamma \ed \loc$,
%the positive excursions of the skew Brownian motion $X_\bgamma$ form a Poisson point process with 
%intensity measure $\frac{1}{2}(1-\bgamma)\linebreak[2] n_{\rm ex}(d\omega)ds$,
%while the  
%%%%%on the time scale of the inverse local times $\tau_\bgamma(s)=\inf\{t\ge 0\colon \ell_\bgamma(t)>s\}$, $s\ge 0$, of $X_\bgamma$. 
%absolute values of negative excursions
%of $X_\bgamma$
 %%%of skew Brownian motion 
%form an independent Poisson point process with intensity measure 
%  $\frac{1}{2}(1+\bgamma)n_{\rm ex}(d\omega)ds$. 
the absolute values of excursions of $X_\bgamma$ away from $0$ into $I^\pm$
form independent Poisson point processes with  intensity
measures $\frac{1}{2}(1- \pm \bgamma)\linebreak[2] n_{\rm ex}(d\omega)ds$.
%while the  
%%%%%on the time scale of the inverse local times $\tau_\bgamma(s)=\inf\{t\ge 0\colon \ell_\bgamma(t)>s\}$, $s\ge 0$, of $X_\bgamma$. 
%absolute values of negative excursions
%of $X_\bgamma$
 %%%of skew Brownian motion 
%form an independent Poisson point process with intensity measure 
%  $\frac{1}{2}(1+\bgamma)n_{\rm ex}(d\omega)ds$. 
Note that $\ell_\bgamma(t)$ splits naturally into the contributions from these positive and negative   
  excursions:
%  \begin{align}\label{ltplus}\lim_{h\downarrow 0}\frac{1}{2h}\int_0^t1_{[0,h]}(X_\bgamma(s))ds&=\frac{1}{2}(1-\bgamma)\ell_\bgamma(t)\\
%    \label{ltminus}\lim_{h\downarrow 0}\frac{1}{2h}\int_0^t1_{[-h,0]}(X_\bgamma(s))ds&=\frac{1}{2}(1+\bbgamma)\ell_\bgamma(t).
%  \end{align}
\begin{align}\label{ltpm}\lim_{h\downarrow 0}\frac{1}{2h}\int_0^t1\{X_\bgamma(s) \in I^{\pm} \cap [-h,h] \}ds&=\frac{1}{2}(1- \pm \bgamma)\ell_\bgamma(t)
%%%\label{ltminus}\lim_{h\downarrow 0}\frac{1}{2h}\int_0^t1_{[-h,0]}(X_\bgamma(s))ds&=\frac{1}{2}(1+\bgamma)\ell_\bgamma(t).
\end{align}
  By Knight's theorem \cite[Theorem V.(1.9)]{RevuzYor}, the time changes 
%$\kappa_\bgamma^+$ and $\kappa_\bgamma^-$ 
$\kappa_\bgamma^\pm$ 
give rise to two independent 
  reflecting Brownian motions 
%$X_\bgamma^+=X_\bgamma\circ\kappa_\bgamma^+$ and $X_\bgamma^-=-X_\bgamma\circ\kappa_\bgamma^-$.
$X_\bgamma^\pm=X_\bgamma\circ\kappa_\bgamma^\pm$.
  %The original source is \cite{knight1971reduction}, where it is applied to prove the independence of the two reflecting Brownian motions that 
  %would appear in our setting for $\bgamma = 0$. 
This argument is detailed in \cite[page 242]{RevuzYor} for the case $\bgamma = 0$, and  %%Essentially the same argument applies for 
extends easily to general $|\bgamma| < 1$. See also the discussion after \cite[Proposition 11]{Lejay}.  

  By these time changes, the local times of 
%$X_\bgamma^+$ and $X_\bgamma^-$ 
$X_\bgamma^\pm$
at 0 are 
%%given by 
the time changes of the limits 
%%  (\ref{ltplus}) and (\ref{ltminus}), 
\eqref{ltpm},
namely
  $\ell_\bgamma^\pm(s)=\frac{1}{2}(1- \pm \bgamma)\ell_\bgamma(\kappa_\bgamma^\pm (s))$.
  %$$\ell_\bgamma^+(s)=\frac{1}{2}(1-\bgamma)\ell_\bgamma(\kappa_\bgamma^+(s))\qquad\mbox{and}\qquad
  %  \ell_\bgamma^-(s)=\frac{1}{2}(1+\bgamma)\ell_\bgamma(\kappa_\bgamma^-(s)).$$
  %%Now let $\bgamma\in(0,1)$. 
We read (\ref{skewbm}) as a decomposition of $B(t)=X_\bgamma(t)+\bgamma \ell_\bgamma(t)$ into excursions away from the   
  increasing process $(\bgamma \ell_\bgamma(t),t\ge 0)$. This increasing process is the inverse of a stable subordinator with Laplace exponent 
  $\sqrt{2\lambda}/\bgamma$. Then %for $\pm \in \{+,-\}$
  %$$W_\bgamma^+(s)=B(\kappa_\bgamma^+(s))=X_\bgamma^+(s)+\bgamma\ell_\bgamma(\kappa_\bgamma^+(s))
  %                                     =X_\bgamma^+(s)+\frac{2\bgamma}{1-\bgamma}\ell_\bgamma^+(s),\quad s\ge 0,$$
  %and
  %$$W_\bgamma^-(s)=-B(\kappa_\bgamma^-(s))=X_\bgamma^-(s)-\bgamma\ell_\bgamma(\kappa_\bgamma^-(s))
  %                                      =X_\bgamma^-(s)-\frac{2\bgamma}{1+\bgamma}\ell_\bgamma^-(s),\quad s\ge 0,$$
  $$W_\bgamma^\pm(s)= \pm B(\kappa_\bgamma^\pm (s))=X_\bgamma^\pm (s)+\bgamma\ell_\bgamma(\kappa_\bgamma^\pm(s)) $$
  %                                     =X_\bgamma^+(s)+\frac{2\bgamma}{1-\bgamma}\ell_\bgamma^+(s),\quad s\ge 0,$$
so that $W_\bgamma^\pm  \ed R_\mu^\pm := |B| \pm  \mu ^\pm \loc$ as in Lemmas \ref{lmplus} and \ref{lmminus} for $\mu^\pm :=2\bgamma/(1- \pm \bgamma)$.
%\end{itemize}
%  \item  $W_\bgamma^- \ed R_\mu^-:= |B| - \mu \loc$ as in Lemma \ref{lmplus} for $\mu:=2\bgamma/(1-\bgamma)$.
\end{proof}
  %is of the same form as $R_\mu^-$ in Lemma \ref{lmminus}, where $\mu=2\bgamma/(1+\bgamma)$.
In this framework of skew Brownian motion, we now give a more explicit statement of the local times decomposition claimed in Theorem \ref{emimm}.
\begin{theorem}\label{emimm2} Let $\delta>0$ and $\bgamma:=1/(1+\delta)$. Let $X_\bgamma$ be skew Brownian motion driven by $B$ as in (\ref{skewbm}), with local 
  time $\ell_\bgamma$ at zero. Let $S_\delta(x)=\inf\{t\ge 0\colon\bgamma\ell_{\bgamma}(t)>x\}$, $x\ge 0$. Then the following two families of random 
  variables are independent
  \begin{itemize}[leftmargin=.7cm]
    \item $Y_0^{(\delta)}:=(L(x,S_\delta(x)),x\ge 0)\ed{\tt BESQ}_0(\delta)$;
    \item $Y_v^\prime:=(L(x,\tau(v))-L(x,S_\delta(x)\wedge\tau(v)),x\ge 0)\ed{\tt BESQ}_v(-\delta)$ for all $v\ge 0$.
  \end{itemize}
  For each $v\ge 0$, the random level $\zeta^\prime(v):=\inf\{x\ge 0\colon S_\delta(x)>\tau(v)\}$ is almost surely finite,
  and coincides with the absorption time of $Y_v^\prime$. Conditionally given $\zeta^\prime(v)=a$, 
  \begin{itemize}[leftmargin=.7cm]
	\item the process $Y_{0,v}^{(\delta)}:=(L(x,S_\delta(x)\wedge\tau(v)),x\ge 0)$ is independent of $Y_v^\prime$ and a time-inhomogeneous Markov process 
       that is ${\tt BESQ}_0(\delta)$ on the time interval $[0,a]$ and then continues as ${\tt BESQ}(0)$.
  \end{itemize}  
\end{theorem}
Note that, with $\bgamma=1/(1+\delta)$, we have $B(S_\delta(x))=X_\bgamma(S_\delta(x))+\bgamma\ell_\bgamma(S_\delta(x))=0+x=x$, since $S_\delta(x)$ is
an inverse local time of $X_\bgamma$. Hence, $S_\delta$ is a right inverse of $B$. Right inverses of L\'evy processes were studied by 
Evans \cite{Evans}, also \cite{Winkel}, to construct stationary local time processes. The main focus has been on the minimal right inverse, which
for $B$ is the first passage process. Theorem \ref{emimm2} involves a family of non-minimal right inverses.    
\begin{proof}[Proof of Theorem \ref{emimm2}] 
%%Let us first sketch the strategy of proof. 
The definition of $S_\delta(x)$ is such that the increasing path in Figure \ref{splitfig}
  is a multiple of the local time of the skew Brownian motion $X_\bgamma$. We write (\ref{skewbm}) as $B(t)=X_\bgamma(t)+\bgamma\ell_\bgamma(t)$. The 
  meaning of this expression is that the positive excursions of $X_\bgamma$ are found in $B$ as excursions 
  above $\bgamma\ell_\bgamma(t)$, while the negative excursions of $X_\bgamma$ are found in $B$ as excursions below $\bgamma\ell_\bgamma(t)$. Recall that $W_\bgamma^+$ 
  and $W_\bgamma^-$ comprise excursions of $B$ above and below $\bgamma\ell_\bgamma$, respectively. 
%The theorem essentially makes claims about their
%  local times. The distributions of these local times were stated in Lemmas \ref{lmplus} and \ref{lmminus}. Here is a more formal proof.  
The theorem identifies the distributions of these local times by application of Lemmas \ref{lmplus} and \ref{lmminus}, as will now
be detailed.

  In the setting of Lemma \ref{propskew}, we can apply Lemma \ref{lmminus} to $W_\bgamma^-$ to see that for all 
  $v>0$, the process $W_\bgamma^-$ up to the inverse $\tau_\bgamma^-(v)=\inf\{s\ge 0\colon L_\bgamma^-(0,s)>v\}$ of the local time 
  $(L_\bgamma^-(0,t),t\ge 0)$ of $W_\bgamma^-$ at zero has two independent local time processes 
  $(L_\bgamma^-(-x,\tau_\bgamma^-(v)),x\ge 0)\ed{\tt BESQ}_v(0)$ and $(L_\bgamma^-(x,\tau_\bgamma^-(v)),x\ge 0)\ed{\tt BESQ}_v(2-2/\mu_\bgamma^-)$, 
  where $\mu_\bgamma^-=2\bgamma/(1+\bgamma)$, i.e. $2-2/\mu_\bgamma^-=1-1/\bgamma=-\delta<0$, since $\bgamma=1/(1+\delta)\in(0,1)$.

  Similarly, Lemma \ref{lmplus} with $a=\infty$, yields that $W_\bgamma^+$ has ultimate local time process 
  $(L_\bgamma^+(x,\infty),x\ge 0)\ed{\tt BESQ}_0(2/\mu_\bgamma^+)$ where $\mu_\bgamma^+=2\bgamma/(1-\bgamma)$, i.e. $2/\mu_\bgamma^+=-1+1/\bgamma=\delta$.

  Let us rewrite the results of the last two paragraphs in terms of the local times $L=(L(x,t),x\in\bR,t\ge 0)$ of $B$. Recall that
  $S_\delta(x)=\inf\{t\ge 0\colon\bgamma\ell_\bgamma(t)>x\}$, $x\ge 0$, where $\bgamma=1/(1+\delta)$, and also set $S_\delta(x):=0$ for $x<0$. Note 
  that 
  $$W_\bgamma^-(A_\bgamma^-(t))\le\bgamma\ell_\bgamma(t)\le W_\bgamma^+(A_\bgamma^+(t))\quad\mbox{and}\quad
    W_\bgamma^-(A_\bgamma^-(t))\le B(t)\le W_\bgamma^+(A_\bgamma^+(t)),$$ 
  where for each $t\ge 0$ and in each of the two statements, at least one of the inequalities is an equality. Let 
  $$\cR_\bgamma^+:=\{(t,x)\in[0,\infty)\times\bR\colon x\ge\bgamma\ell_\bgamma(t)\}=\{(t,x)\in[0,\infty)\times\bR\colon S_\delta(x)\ge t\}.$$
  %and 
  %%$$\cR_\bgamma^-:=\{(t,x)\in[0,\infty)\times\bR\colon x\le\bgamma\ell_\bgamma(t)\}.$$
  Then the occupation measure $U_\bgamma^+$ of $W_\bgamma^+$ can be related to the occupation measure $U$ of $B$ by the usual change of variables 
  $u=\kappa_\bgamma^+(r)$, separately on each excursion interval of $X_\bgamma$ of a positive excursion to give
  \begin{align*}
    U_\bgamma^+([0,s]\times[0,x])&=\int_0^s1_{[0,x]}(W_\bgamma^+(r))dr
                                 =\int_0^{\kappa_\bgamma^+(s)}1_{[0,x]}(B(u))1_{\cR_\bgamma^+}(u,B(u))du\\
                                &=U(([0,\kappa_\bgamma^+(s)]\times[0,x])\cap\cR_\bgamma^+).
  \end{align*}
  By the occupation density formula for $U$ in its general form for time-varying integrands, we obtain 
  $$
    U_\bgamma^+([0,s]\times[0,x])=\int_0^x\int_0^{\kappa_\bgamma^+(s)}1_{\cR_\bgamma^+}(u,y)d_yL(y,u)du
	  						  =\int_0^xL(y,\kappa_\bgamma^+(s)\wedge S_\delta(y))dy.
  $$
  Hence $(L(y,\kappa_\bgamma^+(s)\wedge S_\delta(y)),y\in\bR,s\ge 0)$ is a local time for $W_\bgamma^+$ that is right-continuous in $s$ and in $y$.   
  In particular, we
  deduce from the continuity of $y\mapsto L_\bgamma^+(y,\infty)$ that $L_\bgamma^+(y,\infty)=L(y,S_\delta(y))$ for all $y\ge 0$ almost surely. 
  Similarly, we use the joint continuity of local times of $W_\bgamma^-$ of \cite[Theorem 3.2]{CPY1994} to obtain that
  $L(y,\kappa_\bgamma^-(s)\wedge S_\delta(y))-L(y,S_\delta(y))=L_\bgamma^-(y,s)$ almost surely. In particular, we  
  %Since $\bgamma\ell_\bgamma$ is the inverse of a stable subordinator, we
  have $L(0,t)=L_\bgamma^-(0,A_\bgamma^-(t))$, hence $\tau_\bgamma^-(v)=A_\bgamma^-(\tau(v))$ for all $v\ge 0$, and hence, for all $x\ge 0,v\ge 0$ almost
  surely,
  %Since the space-time local time processes of $B$, $W_\bgamma^+$ and $W_\bgamma^-$ are jointly continuous with probability one, these results  
  %yield, on the same set of probability one that for all $x\ge 0$
  %$$L_\bgamma^+(x,\infty)=L(x,S_\delta(x))\qquad\mbox{and}\qquad L_\bgamma^-(x,\tau_\bgamma^-(v))=L(x,\tau(v))-L(x,\tau(v)\wedge S_\delta(x)).$$
  %Specifically, we can use Borodin's \cite[Theorem 1.3]{Borodin} uniform one-sided approximations of Brownian local time that yield, uniformly for
  %$x\in[0,x_0]$ and $t\ge 0$
  %$$L(x,t\wedge S_\delta(x))\!\underset{h\downarrow 0}{\longleftarrow}\frac{1}{h}\!\int_0^{t\wedge S_\delta(x)}\!\!\!1_{(x,x+h]}(B(s))ds
  %				  =\frac{1}{h}\!\int_0^{A_\bgamma^+(S_\delta(x))}\!\!\!1_{(x,x+h]}
  %                (W_\bgamma^+(s))ds\underset{h\downarrow 0}{\longrightarrow}\!L_\bgamma^+(x,\infty),$$
  %and similarly
  $$L(x,\tau(v))-L(x,\tau(v)\wedge S_\delta(x))  %=\lim_{h\downarrow 0}\frac{1}{h}\int_{\tau(v)\wedge S_\delta(x)}^{\tau(v)}1_{[x-h,x)}(B(s))ds
                                                  =L_\bgamma^-(x,\tau_\bgamma^-(v)).$$
  %denote by $d_\bgamma(t)=\inf\{u>0\colon X_\bgamma(u)=0$ the right endpoint of the excursion straddling (or starting at) $t\ge 0$. Note that
  %$t\mapsto d_\bgamma(t)$ is c\`adl\`ag a.s.. By construction, we have 
  %$$\kappa_\bgamma^+(A_\bgamma^+(t))=\left\{\begin{array}{ll}t&\mbox{if }X_\bgamma(t)>0,\\ d_\bgamma(t)&\mbox{if }X_\bgamma(t)<0,\end{array}\right.
  %  \quad\mbox{and}\quad
  %  \kappa_\bgamma^-(A_\bgamma^-(t))=\left\{\begin{array}{ll}t&\mbox{if }X_\bgamma(t)<0,\\ d_\bgamma(t)&\mbox{if }X_\bgamma(t)>0.\end{array}\right.$$
  %This implies that we can write
  %$$B=X_\bgamma+\bgamma\ell_\bgamma=W_\bgamma^+\circ A_\bgamma^+-X_\bgamma^-\circ A_\bgamma^-=X_\bgamma^+\circ A_\bgamma^+-W)\bgamma^-\circ A_\bgamma^-.$$
  %Since $\bgamma\ell_\bgamma(t)>0$ if and only if $t>0$, as the inverse of a stable subordinator, 
  %Then 
  %$$W_\bgamma^+(s)=X_\bgamma^+(s)+\bgamma\ell_\bgamma(\kappa_\bgamma^+(s))>x\quad\mbox{for all }s>A_\bgamma+(S_\delta(x)),$$
  %and since $\bgamma\ell_\bgamma^+(S_\delta)$
  
  To complete the proof, we consider the random level $\zeta^\prime(v)=\inf\{x\ge 0\colon L(x,\tau(v))=L(x,S_\delta(x)\wedge\tau(v))\}$. Since
  $L$ and $S_\delta$ are both increasing and $B(\tau(v))\!=\!0$ while $B(S_\delta(x))\!=\!B(S_\delta(x-))\!=\!x$, we can only have 
  $L(x,\tau(v))=L(x,S_\delta(x)\wedge\tau(v))$ if $S_\delta(x)>\tau(v)$, and $\zeta^\prime(v)=\inf\{x\ge 0\colon S_\delta(x)>\tau(v)\}$. Note that 
  $\zeta^\prime(v)=\inf\{x\ge 0\colon L_\bgamma^-(x,\tau_\bgamma^-(v))=0\}$, as a function of $W_\bgamma^-$, is independent 
  of $W_\bgamma^+$. Conditionally given $\zeta^\prime(v)=a$, we can apply Lemma \ref{lmplus} to obtain that
  $(L_\bgamma^+(x,A_\bgamma^-(\tau(v))),x\ge 0)=(L(x,S_\delta(x)\wedge\tau(v)),x\ge 0)$ also has the desired distribution. 
\end{proof}

\begin{proof}[Proof of Corollary \ref{cor10}] Let $\delta>\delta^\prime>0$. Let $\Upsilon_{\rm top}=(L(x,S_{\delta^\prime}(x)),x\ge 0)$, $\Upsilon_{\rm mid}=(L(x,S_{\delta}(x))\!-\!L(x,S_{\delta^\prime}(x)),x\!\ge\! 0)$
  and $\Upsilon_{\rm rest}\!=\!(L(x,\tau(v))\!-\!L(x,S_{\delta}(x)\!\wedge\!\tau(v)),x\!\ge\! 0,v\!\ge\! 0)$. By the proof of Theorem \ref{emimm2}, we have
  $(\Upsilon_{\rm top},\Upsilon_{\rm mid})$ independent of $\Upsilon_{\rm rest}$, and we have $\Upsilon_{\rm top}$ independent of 
  $(\Upsilon_{\rm mid},\Upsilon_{\rm rest})$. Hence, $\Upsilon_{\rm mid}$ is independent of $\Upsilon_{\rm top}$. By additivity,
  $\Upsilon_{\rm mid}\ed{\tt BESQ}_0(\delta-\delta^\prime)$, independent of $\Upsilon_{\rm top}\ed{\tt BESQ}(\delta)$. A straightforward
  induction completes the proof.
\end{proof}

\section{Some checks on Proposition \ref{prop:add}}\label{sec:checks}

To simplify presentation, for any real $\delta$ and $v\ge 0$ we will denote by ${\tt BESQ}^{(\delta)}_v=({\tt BESQ}^{(\delta)}_v(x),x\ge 0)$ a process with law ${\tt BESQ}_v(\delta)$.
For $r \ge 0$ let $\gamma(r)$ denote a gamma variable, with $\gamma(0) = 0$ and 
\begin{equation}
\label{gammadens}
\frac{ \bP(\gamma(r)\in dt)}{dt} = f_r(t):= \frac{1} {\Gamma(r)} t^{r-1}e^{-t}1(t>0).
\end{equation}
Fix $x >0$.  
To check the implication of Corollary \ref{propadd} that $Y(x)+Y^\prime(x) \ed {\tt BESQ}^{(\delta)}_v(x)$ for all $v\ge 0$,
by uniqueness of Laplace transforms it suffices to show for all $\mu >0$ that in the modified setting where $Y^\prime\ed {\tt BESQ}_{\gamma(1)/\mu}^{(-\delta)}$, meaning that $Y(0)$ is assigned the
exponential distribution of $\gamma(1)/\mu$, that $Y(x)+Y^\prime(x) \ed {\tt BESQ}_{\gamma(1)/\mu}^{(0)}(x)$.
To show this, first recall some known facts:
\begin{lemma} 
Let $\delta \ge 0$. Then
\begin{enumerate}[leftmargin=.7cm]
  
\item[ {\em (a)}]
${\tt BESQ}^{(\delta)}_0(x)\ed 2x\gamma(\delta/2)$.% and ,\\ and hence
%             $R\sim{\tt BESQ}_{2m\gamma(\delta/2)}(\delta)\Rightarrow R(x-m)\ed 2x\gamma(\delta/2)$, for any $\delta>0$ and $0 \le m \le x$;
  
  \item[{\em (b)}] ${\tt BESQ}^{(-\delta)}_{\gamma(1)/\mu}(x)\ed (2\mu x+1)\gamma(1)I(1/(2\mu x+1)^{1 + \delta/2})$
where $\gamma(1)$ is independent of the indicator variable $I(p)$ with Bernoulli $(p)$ distribution for $p = 1/(2\mu x+1)^{1 + \delta/2}$.
%  \item[(c)] $R\sim{\tt BESQ}_{\gamma(1)/\mu}(-\delta)\Rightarrow\zeta_R:=\inf\{x\ge 0\colon R(x)=0\}\ed\gamma(1)/2\mu\gamma^\prime(1+\delta/2)$ 
%     for all $\delta\ge 0$, and hence 
%       $\bP(\zeta^\prime>x)=\bP(\gamma(1)/2\mu x>\gamma^\prime(1+\delta/2)=(1/(2\mu x+1))^{1+\delta/2}$ and 
%         $\bP(\zeta^\prime\in dm)=(\delta+2)\mu/(2\mu m+1)^{2+\delta/2}dm$;
\end{enumerate}
\end{lemma}
Here (a) is a consequence of the additivity property \eqref{fulladd}, 
while (b) details for $-\delta = 2 - 2 \alpha \le 0$  the entrance law for ${\tt BESQ}(2 - 2 \alpha)$ killed at $T_0$, with $\alpha >0$, 
which was identified by \cite[(3.2) and (3.5)]{PitmYor82}. 
The case of (b) for $\delta=0$ and $\mu = 1/(2 b)$ is also an easy consequence of the Ray-Knight description of Brownian local times
$L(x, T_{-b}), x \ge 0 ) \ed {\tt BESQ}^{(0)}_{ 2 b \gamma(1)}$.
Applying this instance of (b), we find that ${\tt BESQ}_{\gamma(1)/\mu}^{(0)}(x)$ has Laplace transform (in $\lambda$)
\begin{equation}\label{LHS}
\left(1-\frac{1}{2\mu x+1}\right)
+
\frac{1}{2\mu x+1}\,\frac{1}{1+\lambda(2\mu x+1)/\mu}
=\frac{(2\lambda x+1)\mu}{(2\lambda x+1)\mu+\lambda}.
\end{equation}
%%where the first term comes from strictly positive values, and the second term from the zero value.
On the other hand, we obtain the Laplace transform of $Y(x)+Y^\prime(x)$ by conditioning $Y(x)$ and $Y^\prime(x)$ on 
$\zeta^\prime=\inf\{x\ge 0\colon Y^\prime(x)=0\}$. Specifically, now using (b) for $Y^\prime$ and (a) for $Y$, we find
\begin{align}\bE\left(\exp\left(-\lambda(Y(x)+Y^\prime(x))\right)\right)
  =&\frac{1}{(2\mu x+1)^{1+\delta/2}}\,\frac{1}{1+\lambda(2\mu x+1)/\mu}\,\frac{1}{(1+\lambda 2x)^{\delta/2}}\nonumber\\
    &+\int_0^x\frac{(\delta+2)\mu}{(2\mu m+1)^{2+\delta/2}}\bE\left(e^{-\lambda{\tt BESQ}^{(0)}_{2m\gamma(\delta/2)}(x-m)}\right)dm.\label{integr}
\end{align}
By the additivity property of ${\tt BESQ}(0)$ and then proceeding as for \eqref{LHS}, 
we have
$$\bE\left(e^{-\lambda{\tt BESQ}^{(0)}_{2m\gamma(\delta/2)}(x-m)}\right)=\left(\bE\left(e^{-\lambda{\tt BESQ}_{2m\gamma(1)}^{(0)}(x-m)}\right)\right)^{\delta/2}
   =\left(\frac{1+2(x\!-\!m)\lambda}{1+2\lambda x}\right)^{\delta/2}.$$
The change of variables $m = x u$, $dm = x du$ allows the integral in (\ref{integr}) to be expressed as
$$\frac{(\delta+2)\mu x}{(1+2\lambda x)^{\delta/2}}\int_0^1\frac{(1+2\lambda x(1-u))^{\delta/2}}{(1+2\mu xu)^{2+\delta/2}}du.$$
Writing $a=2\lambda x$, $b=2\mu x$ and $q=1+\delta/2$, this integral is of the form
\begin{equation}\label{Wolf}\int_0^1\frac{(1+a(1-u))^{q-1}}{(1+bu)^{q+1}}du=\frac{(1+a)^q-(1+b)^{-q}}{(a+ab+b)q}
\end{equation}
where the integral is evaluated as $F(1) - F(0)$ 
for the indefinite integral
%%While Wolfram alpha does not give the expression for $(a,b,q)$, they provide indefinite integrals for $(3,5,q)$, $(3,7,q)$ and $(5,7,q)$, from 
%%hich the general indefinite integral is easily guessed to be 
$$F(u)=-\frac{(1+a(1-u))^q(1+bu)^{-q}}{(a+ab+b)q}.$$
The identification of \eqref{LHS} and \eqref{integr} is now elementary using (\ref{Wolf}). 

Consider next the distribution of $\int_0^\infty(Y+Y^\prime)(x)dx$ in the setting of Corollary \ref{propadd}. By 
the corollary, this is the distribution of the corresponding integral of a ${\tt BESQ}_v(0)$, which according to the Ray--Knight theorem for local
times of $B$ at time $\tau(v)$ is that of 
$$\tau_+(v):=\int_0^{\tau(v)}1_{\{B_+>0\}}dt\overset{d}{=}\tau(v/2).$$
The equivalent equality of Laplace transforms at $\frac{1}{2}\lambda^2$ reads
\begin{equation}
\label{stableint}
\bE\left(\exp\left(-\frac{1}{2}\lambda^2\int_0^\infty(Y+Y^\prime)(x)dx\right)\right)=\exp\left(-\frac{v}{2}\lambda\right).
\end{equation}
This formula too can be checked from the construction of $Y$ and $Y^\prime$ by conditioning on $\zeta^\prime$.
For the ${\tt BESQ}_0(\delta)$ process $Y$ on $[0,m]$ continued as ${\tt BESQ}_{Y(m)}(0)$, we have
\begin{equation}\label{intY}
  \bE\left(\left.\exp\left(-\frac{1}{2}\lambda^2\int_0^\infty Y(x)dx\right)\,\right|\,\zeta^\prime=m\right)=\exp\left(-\frac{1}{2}\delta m\lambda\right),
\end{equation}
by Lemma \ref{lmplus} applied with $a=m$ and $\mu=2/\delta$, since these substitutions make
$$\left(\left.\int_0^\infty Y(x)dx\,\right|\,\zeta^\prime=m\right)\overset{d}{=}\int_0^\infty L_\mu^+(x,\tau(a/\mu))dx=\tau(a/\mu)$$
with $a/\mu=\frac{1}{2}\delta m$. On the other hand, given $\zeta^\prime=m$, 
an
application of the formula of \cite[Proposition (5.10)]{PitmYor82} yields the first passage bridge functional
\begin{align}
  &\bE\left(\left.\exp\left(-\frac{1}{2}\lambda^2\int_0^mY^\prime(x)dx\right)\,\right|\,\zeta^\prime=m\right)\nonumber\\
    &=\left(\frac{\lambda m}{\sinh(\lambda m)}\right)^{(4+\delta)/2}\exp\left(-\frac{v}{2m}\left(\lambda m\coth(\lambda m)-1\right)\right).
  \label{intYprime}
\end{align}
To complete the calculation of the left side of \eqref{stableint} 
we must integrate the product of expressions in (\ref{intY}) and (\ref{intYprime}) with respect to the distribution of $\zeta^\prime$, the absorption
time of ${\tt BESQ}_v(-\delta)$, which is the distribution of $v/(2\gamma(1+\delta/2))$ with density
$$\frac{\bP(\zeta^\prime\in dm)}{dm}=\frac{v}{2m^2}f_{1+\delta/2}\left(\frac{v}{2m}\right)$$
where $f_{r}(t)$ is the gamma$(r)$ density at $t$ as  in  \eqref{gammadens}.
So at the level of the total integral functional, the identification of the distribution of $Y+Y^\prime$ as ${\tt BESQ}_v(0)$ implies the identity
$$
  \exp\left(\!-\frac{1}{2}\lambda v\!\right)\!
    %&=\!\int_0^\infty\frac{v}{2m^2}\frac{1}{\Gamma(1+\delta/2)}\left(\frac{v}{2m}\right)^{\delta/2}\left(\frac{\lambda m}{\sinh(\lambda m)}\right)^{(4+\delta)/2}\\
	%			&\qquad\qquad\exp\left(-\frac{v}{2m}-\frac{1}{2}\delta\lambda m-\frac{v}{2m}\left(\lambda m\coth(\lambda m)-1\right)\right)dm\\
	=\!\int_0^\infty\!\frac{v^{1+\delta/2}b^{2+\delta/2}}{2^{\delta/2}\Gamma(1+\delta/2)}\left(\frac{1}{\sinh(\lambda m)}\right)^{2+\delta/2}
				\!\exp\left(-\frac{1}{2}\delta\lambda m-\frac{v\lambda}{2}\coth(\lambda m)\!\right)dm.
$$				
Make the change of variables $x=\lambda m$, $dm=dx/\lambda$, then set $t=\lambda v/2$, $\delta=2p$, to see that this evaluation shows that
Corollary \ref{propadd} has the following consequence.
\begin{corollary} For all $t>0$ and $p\ge 0$,
  $$\int_0^\infty\frac{\exp(-px-t\coth(x))}{(\sinh(x))^{2+p}}dx=\frac{\Gamma(1+p)e^{-t}}{t^{1+p}}.$$ 
\end{corollary}
The simplest case of the Corollary is for $p=0$. Then it is some variation of Knight's analysis of the joint distribution of $\tau(v)$ and
$M(\tau(v)):=\max\{|B(s)|,0\le s\le\tau(v)\}$. See Section 11.3 of \cite{YenYor}, especially formula (11.3.1). 
For $p=0$ or $p=1$, the integrals are easily evaluated using the elementary indefinite integrals
\begin{align*}\int\frac{e^{-t\coth(x)}}{\sinh^2(x)}dx&=\frac{e^{-t\coth(x)}}{t};\\
%%  \mbox{and}
\quad\int\frac{e^{-x-t\coth(x)}}{\sinh^3(x)}dx&=\frac{e^{-t\coth(x)}(1-t+t\coth(x))}{t^2}.
\end{align*}
According to {\em Mathematica}, there are similar expressions for $p=2,3,\ldots$, but they get more complicated and their general structure is not 
readily apparent. 
%Mathematica confirms the definite integral symbolically for $p=0,1,2,3,4,5$. Mathematica is unable to give an indefinite 
%integral for general $p$ or even for $p=1/2$, which is $\delta=2p=1$; also, no definite integral except for integer $p$. But numerical
%evaluations work well. 

This kind of argument can be extended to a full proof of Proposition \ref{prop:add} by the method of 
%%%Pitman and Yor \cite{PitmYor82}, that is 
computating the Laplace functional
\begin{equation}\label{laplace}\bE\left(\exp\left(-\int_0^\infty Z(x)\rho(dx)\right)\right)
\end{equation}
for suitable measures $\rho$ on $(0,\infty)$, and showing  that it equals the known Laplace functional of ${\tt BESQ}_{y+y^\prime}(\delta-\delta^\prime)$, found in \cite{PitmYor82} when $\delta-\delta^\prime\ge 0$, for enough such $\rho$. 
Let us briefly sketch 
this here for the (slightly easier) case when $\delta+\delta^\prime\ge 0$ and $y=0$, $y^\prime=v$. Without loss of generality, $\delta^\prime>0$, 
as otherwise the statement follows from full additivity. The case $y>0$ is then also straightforward, while $\delta-\delta^\prime<0$ 
follow similarly. We claim that for any function $f\colon[0,\infty)\rightarrow[0,\infty)$ that is Lebesgue 
integrable on $[0,z]$ for all $z\ge 0$, the Laplace functional (\ref{laplace}) reduces to the known Laplace functional of ${\tt BESQ}_v(\delta-\delta^\prime)$ 
when $\rho(dx)=f(x)dx$. By \cite[Theorem XI.(1.7)]{RevuzYor}, the latter is
\begin{equation}\label{RHS}(\phi(\infty))^{(\delta-\delta^\prime)/2}\exp\left(\frac{v}{2}\phi^\prime(0)\right)
\end{equation}
where $\phi$ is the unique positive, non-increasing solution to the Sturm--Liouville equation
\begin{equation}\label{SL}\phi^{\prime\prime}=2f\phi,\qquad\phi(0)=1.
\end{equation}
This solution is 
%%%(weakly) 
convex and converges at $\infty$ to some $\phi(\infty)\in(0,1]$. We compute the Laplace functional in the 
setting of Proposition \ref{prop:add} by conditioning first on $\zeta^\prime(v)=x$ and
then on $Y(x)=y$. We use notation $\bP_y^{(\gamma)}:={\tt BESQ}_y(\gamma)$ and also $\bP_{a,b}^{(\gamma),x}$ for the distribution of
a ${\tt BESQ}(\gamma)$ bridge of length $x$ from $a\ge 0$ to $b\ge 0$. Specifically, for $\delta^\prime>0$, for each $a >0$ we define $\bP_{a,0}^{(-\delta'),x}$ for $x \ge 0$ 
to be the first passage bridge obtained as the weakly continuous conditional distribution of $\bP^{(-\delta^\prime)}_a(\,\cdot\,|\,T_0=x)$. By duality, 
this equals $\bP^{(4+\delta^\prime),x}_{a,0}$, which is the time reversal of $\bP^{(4+\delta^\prime),x}_{0,a}$. See 
\cite{PitmYor82}. We need several expectations of quantities of the form $\cL(f,w):=\exp(-\int_0^wY(u)f(u)du)$ and also use notation $\theta_x(f)=f(x+\cdot)$. In this notation, we want to compute
\begin{equation}\label{goal}
\int_0^\infty\!\bP_{v,0}^{(-\delta^\prime),x}(\cL(f,x))\!\left(\int_0^\infty\!\bP_{0,y}^{(\delta),x}(\cL(f,x))\bP_y^{(\delta-\delta^\prime)}(\cL(\theta_xf,\infty))\bP(Y(x)\!\in\! dy)\!\right)\bP(\zeta^\prime(v)\!\in\! dx).
\end{equation}
The key technical formula is a generalisation of \cite[Theorem XI.(3.2)]{RevuzYor} from unit-length bridges to bridges of length $x$, which we 
express in terms of the solution of $\phi^{\prime\prime}= 2 f\phi$ without restricting $f$ to support in $[0,x]$. We obtain for all $\gamma>0$, $a\ge 0$, $b\ge 0$, $x>0$
\begin{equation}\label{bridgefunctional}
  \bP_{a,b}^{(\gamma),x}(\cL(f,x))
  =\left(\phi(x)\right)^{\gamma/2}\exp\left(\frac{a}{2}\phi^\prime(0)-\frac{b}{2}\,\frac{\phi^\prime(x)}{\phi(x)}\right)
                                \frac{q^{(\gamma)}_{\sigma^2(x)}(a(\phi(x))^2,b)}{q_x^{(\gamma)}(a,b)},
\end{equation}
where $\sigma^2(x)=(\phi(x))^2\int_0^x(\phi(u))^{-2}du$, and $q_u^{(\gamma)}(v,y)=\bP_v^{(\gamma)}(Y(u)\in dy)/dy$ is the continuous ${\tt BESQ}(\gamma)$ transition density on $(0,\infty)$. Using $\bP_{v,0}^{(-\delta^\prime),x}=\bP_{v,0}^{(4+\delta^\prime),x}$ for the first, this 
yields the two bridge functionals of (\ref{goal}) while the remaining functional can be obtained from \cite[Theorem XI.(1.7)]{RevuzYor}. We leave 
the remaining details to the reader.

%\subsection{Fixed marginals}

%xxx this is a minimal editing of Matthias's draft of the same section, with only minor edits for clarity, 

\section{Related developments in the literature}\label{sec:lit}

%\subsection{Squared Bessel total mass processes of interval partition diffusions}\label{sec:FPRW}

We discuss three related developments in the literature. These are interval partition 
diffusions, an instance of a weaker form of additivity related to sticky Brownian motion, and local time flows generated by skew Brownian motion.

%Let us recall some of the framework of \cite{Paper1}. {\tt to be continued...} See Figure \ref{FPRWfig}.

Motivated by Aldous's conjectured diffusion on a space of continuum trees, the authors of \cite{Paper1} study interval partition diffusions in 
which interval lengths evolve as independent ${\tt BESQ}(-1)$ processes until absorption at 0, while new intervals are created according to a 
Poisson point process of ${\tt BESQ}(-1)$ excursions. See also \cite{Paper3} and further references there. Specifically, 
the construction for one initial interval of length $v$ is illustrated in Figure \ref{FPRWfig} (left). Next to a ${\tt BESQ}_v(-1)$ process 
$Y^\prime$ with absorption time $\zeta^\prime(v)$, the total sums of all other interval lengths form a process $Y$ that is shown to be 
${\tt BESQ}_0(1)$ up to $\zeta^\prime(v)$ continuing as ${\tt BESQ}(0)$, as in Corollary \ref{propadd}. See 
\cite[Theorem 1.5 and Corollary 5.19]{Paper1}. 
A generalisation to ${\tt BESQ}(-\delta)$ for $\delta\in(0,2)$ is indicated in \cite[Section 6.4]{Paper0}, to be taken up elsewhere.

\begin{figure}[t]
  \begin{picture}(415,100)
          \put(0,0){\includegraphics[scale=0.44]{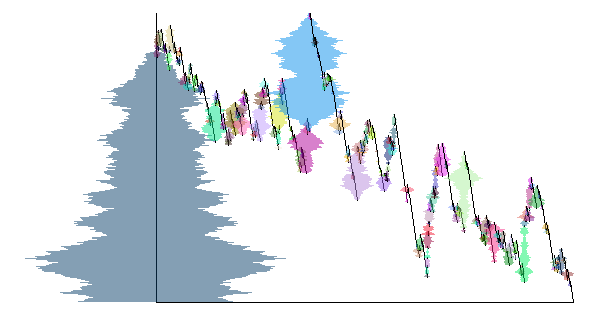}}\put(205,0){\includegraphics[height=100pt,width=200pt]{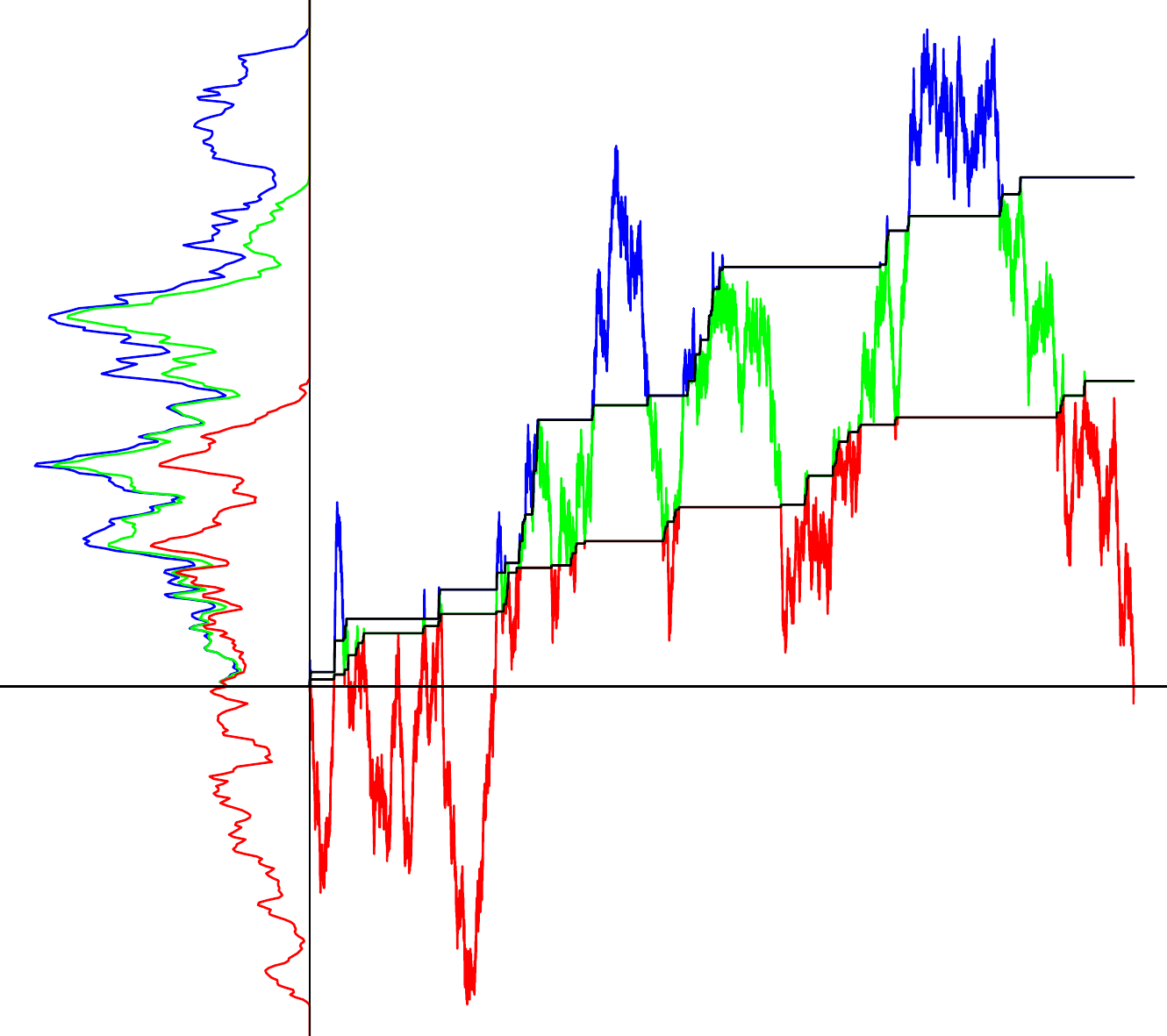}}
          \put(40,88){\line(1,0){145}}
          \put(20,88){$\scriptstyle\zeta^\prime(v)$}
          \put(49.5,-2){$\scriptstyle v$}
          \put(47.5,0){\vector(-1,0){21.5}}
          \put(55.5,0){\vector(1,0){21.5}}
%          \put(370,40){$\scriptstyle B(t)$}
%          \put(375,90){$\scriptstyle t$}
%          \put(130,145){$\scriptstyle\zeta^\prime(v)$}
          \put(35,15){$\scriptstyle {\tt BESQ}_v(-\delta)$}
          \put(148,78){$\scriptstyle {\tt BESQ}_0(\delta)$}
%          \put(65,145){\vector(4,-3){20}}
          \put(150,92){$\scriptstyle {\tt BESQ}(0)$}
%          \put(90,165){\vector(4,-1){35}}
%          \put(90,70){$\scriptstyle {\tt BESQ}_v(0)$}
  \end{picture}
  \caption{Left: The left-most shaded area is $Y^\prime\!\sim\!{\tt BESQ}_v(-\delta)$ for some $\delta\!\in\!(0,2)$, represented 
    in the widths of a symmetric ``spindle'' shape. Other shaded areas form a Poisson point process of ${\tt BESQ}(-\delta)$ excursions 
    placed on a ``scaffolding'', an induced ${\tt Stable}(1\!+\!\delta/2)$ process whose jump heights are the excursion lifetimes%
%; as a 
%    function of $x\ge 0$, the sums of widths $Y(x)$ at level $x$ form a ${\tt BESQ}_0(1)$ process up to level $\zeta^\prime(v)$ and then
%    continue as ${\tt BESQ}(0)$
. Simulation, as in \cite{Paper3}, due to \cite{WXMLCRP}.
    Right: Simulation of Brownian motion split along two increasing paths. \label{FPRWfig}}
\end{figure}
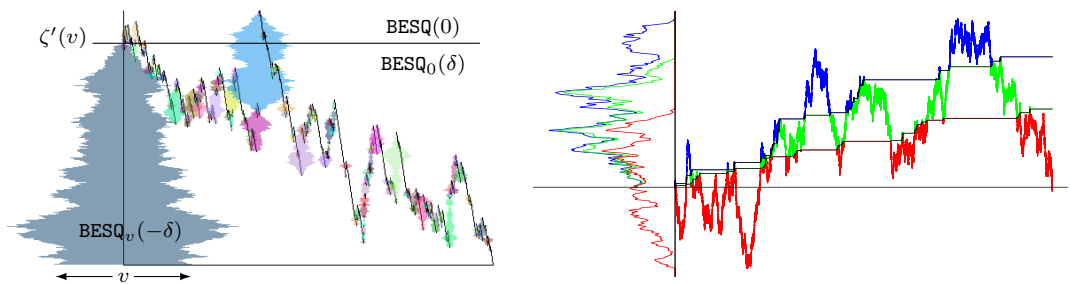

%\subsection{Additivity of squared Bessel and related stochastic processes}\label{sec:add}

Shiga and Watanabe \cite{ShiWat73} %%%showed the additivity of squared Bessel processes in a paper exploring additivity properties for one-dimensional diffusions. 
showed that families of one-dimensional diffusions with the additivity property
can be parameterised by three real parameters, one of which corresponds to a linear time-change parameter affecting the diffusion 
coefficient, which we fix here without loss of generality. The family formed by the other two parameters, $\delta$ and $\mu$, are the strong 
solutions to the stochastic differential equation
\begin{equation}\label{genbesqdef}
  dY(t)=(\delta-\mu Y(t))\,dt+2\sqrt{Y(t)}dB(t),\quad Y(0)=y,\quad y\ge 0,\qquad \delta\ge 0,\mu\in\bR.
\end{equation}
We observe that these families may be extended to $\delta < 0$ with absorption at $0$ much as in  \eqref{besqdef},
and that the statement and proof of Proposition \ref{prop:add} generalizes straightforwardly to this case.
%We denote the distribution of the strong solution by ${\tt BESQ}_y(\delta,\mu)$. We state the weaker form of additivity analogous to 
%\begin{proposition} For any real $\delta,\delta^\prime,\mu$ and $y,y^\prime\ge 0$, let $Y$ and $Y^\prime$ be independent with 
%\begin{itemize}[leftmargin=.7cm]
%\item $Y$ is ${\tt BESQ}_y(\delta,\mu)$ with lifetime $\zeta$;
%\item $Y'$ is ${\tt BESQ}_{y^\prime}(\delta^\prime,\mu)$ with lifetime $\zeta'$.
%\end{itemize}
%Let $T$ be a stopping time relative to the filtration generated by the pair $(Y,Y')$, with $T \le \zeta \wedge \zeta'$. Define 
%$Z(x):=Y(x)+Y'(x)$ if $0\le x\le T$, and given $(Z(x),0\le x\le T)$, let $Z$ continue as a ${\tt BESQ}_{Z(T)}(\delta+\delta^\prime,\mu)$.  
%Then $Z$ is a ${\tt BESQ}_{y+y^\prime}(\delta+\delta^\prime,\mu)$. 
%\end{proposition}
Warren \cite[Proposition 3]{Warren1997} establishes an additivity in the case $\delta=0$ that involves the parameter $\mu$ of (\ref{genbesqdef}), 
where the second process  
$Y^\prime=(Y^\prime(t),t\ge 0)$ is driven by a Brownian motion $(B^\prime(t),t\ge 0)$ independent of the first process $Y$, but the first process $(Y(t),t\ge 0)$ appears in the $dt$-part of its 
stochastic differential equation: 
$$dY^\prime(t)=\mu Y(t)\,dt+2\sqrt{(Y^\prime(t))^+}dB^\prime(t),\qquad Y^\prime(0)=0.$$
Specifically, $Y+Y^\prime\ed{\tt BESQ}_y(0)$. Furthermore, \cite[Theorems 10-11 and Proposition 12]{Warren1997} demonstrate how to find this 
decomposition embedded in the local times of a given Brownian motion, using the Brownian motion to drive a stochastic differential equation whose 
strong solution is sticky Brownian motion of parameter $\mu\ge 0$.

%\subsection{Other related literature}

Burdzy et al.  \cite{MR1880238,%%, AUTHOR = {Burdzy, Krzysztof and Chen, Zhen-Qing}, TITLE = {Local time flow related to skew {B}rownian motion}, JOURNAL = {Ann. Probab.},
MR2094439} %, AUTHOR = {Burdzy, Krzysztof and Kaspi, Haya}, TITLE = {Lenses in skew {B}rownian flow},
treat other aspects of what they call the local time flow generated by skew {B}rownian motion. They study solutions to uncountably many coupled variants of (\ref{skewbm}) jointly. Specifically, \cite{MR2094439} focusses on 
$(X_\gamma^{s,x}(t),L_\gamma^{s,x}(t))$, $t\ge s$, $x\in\bR$, for $B(t)$ in (\ref{skewbm}) replaced by $x+B(t)-B(s)$, while \cite{MR1880238} exhibits various  
one-dimensional families indexed by $x$ or by $\gamma$ that form Markov processes in a way reminiscent of Ray--Knight theorems. 

Taking $s=0$, one viewpoint is to read these coupled solutions as joint 
decompositions of $B(t)=X_\gamma^{s,x}(t)-x+\gamma L_\gamma^{s,x}(t)$ along increasing paths $-x+\gamma L_\gamma^{s,x}(t)$. Consider the 
coupling in $\gamma\in(0,1)$ of \cite[Theorem 1.3 and 1.4]{MR1880238} when $x=0$. They note for 
$\gamma_1<\gamma_2$ that $\gamma_1\ell_{\gamma_1}(t)\le\gamma_2\ell_{\gamma_2}(t)$ for all $t\ge 0$, cf. Figure \ref{FPRWfig} (right), and they establish a phase transition
when $\gamma_1=\gamma_2/(1+2\gamma_2)$. By our Corollary \ref{cor10}, applied to an increment 
$Y^{(\delta_1)}_0-Y^{(\delta_2)}_0\ed{\tt BESQ}(\delta_1-\delta_2)$, we identify the 
same phase transition with the behaviour around the critical dimension $\delta=2$ of ${\tt BESQ}_0(\delta)$, since for $\delta_i=-1+1/\gamma_i$, $i=1,2$, we have $\gamma_1=\gamma_2/(1+2\gamma_2)$ if and only if $\delta_1-\delta_2=2$.

%\subsection{Feynman--Kac approach}\label{sec:FK}

%Feynman--Kac xxx

\bibliographystyle{abbrv}
\bibliography{negbes} 

\begin{thebibliography}{10}

\bibitem{Alexander2011}
K.~Alexander.
\newblock Excursions and local limit theorems for bessel-like random walks.
\newblock {\em Electron. J. Probab.}, 16:1--44, 2011.

\bibitem{BerKor2016}
J.~Bertoin and I.~Kortchemski.
\newblock {Self-similar scaling limits of Markov chains on the positive
  integers}.
\newblock {\em Ann. Appl. Probab.}, 26(4):2556--2595, 08 2016.

\bibitem{MR1880238}
K.~Burdzy and Z.-Q. Chen.
\newblock Local time flow related to skew {B}rownian motion.
\newblock {\em Ann. Probab.}, 29(4):1693--1715, 2001.

\bibitem{MR2094439}
K.~Burdzy and H.~Kaspi.
\newblock Lenses in skew {B}rownian flow.
\newblock {\em Ann. Probab.}, 32(4):3085--3115, 2004.

\bibitem{CPY1994}
P.~Carmona, F.~Petit, and M.~Yor.
\newblock Some extensions of the arc sine law as partial consequences of the
  scaling property of {B}rownian motion.
\newblock {\em Probab. Theory Related Fields}, 100(1):1--29, 1994.

\bibitem{Evans}
S.~N. Evans.
\newblock Right inverses of {L}\'evy processes and stationary stopped local
  times.
\newblock {\em Probab. Theory Related Fields}, 118(1):37--48, 2000.

\bibitem{WXMLCRP}
N.~Forman, G.~Brito, Y.~Chou, A.~Forney, and C.~Li.
\newblock {WXML} final report: {C}hinese restaurant process.
\newblock 2017.

\bibitem{Paper1}
N.~Forman, S.~Pal, D.~Rizzolo, and M.~Winkel.
\newblock Diffusions on a space of interval partitions with {Poisson-Dirichlet}
  stationary distributions.
\newblock arXiv: 1609.06706v2, 2017.

\bibitem{Paper3}
N.~Forman, S.~Pal, D.~Rizzolo, and M.~Winkel.
\newblock {Interval partition evolutions with emigration related to the Aldous
  diffusion}.
\newblock arXiv:1804.01205 [math.PR], 2018.

\bibitem{Paper0}
N.~Forman, S.~Pal, D.~Rizzolo, and M.~Winkel.
\newblock Uniform control of local times of spectrally positive stable
  processes.
\newblock {\em To appear in Ann. Appl. Probab.}, 2018.
\newblock arXiv: 1609.06707.

\bibitem{GoinYor03}
A.~G{\"o}ing-Jaeschke and M.~Yor.
\newblock A survey and some generalizations of {B}essel processes.
\newblock {\em Bernoulli}, 9(2):313--349, 2003.

\bibitem{HarrShep81}
J.-M. Harrison and L.-A. Shepp.
\newblock {On skew Brownian motion}.
\newblock {\em Ann. Probab.}, 9(2):309--313, 1981.

\bibitem{KawWat71}
K.~Kawazu and S.~Watanabe.
\newblock Branching processes with immigration and related limit theorems.
\newblock {\em Teor. Verojatnost. i Primenen.}, 16:34--51, 1971.

\bibitem{Lambert2002}
A.~Lambert.
\newblock The genealogy of continuous-state branching processes with
  immigration.
\newblock {\em Probab. Theory Related Fields}, 122(1):42--70, 2002.

\bibitem{LGY1986}
J.-F. Le~Gall and M.~Yor.
\newblock Excursions browniennes et carr\'es de processus de {B}essel.
\newblock {\em C. R. Acad. Sci. Paris S\'er. I Math.}, 303(3):73--76, 1986.

\bibitem{MR2815763}
Y.~Le~Jan.
\newblock {\em Markov paths, loops and fields}, volume 2026 of {\em Lecture
  Notes in Mathematics}.
\newblock Springer, Heidelberg, 2011.
\newblock Lectures from the 38th Probability Summer School held in Saint-Flour,
  2008, \'Ecole d'\'Et\'e de Probabilit\'es de Saint-Flour. [Saint-Flour
  Probability Summer School].

\bibitem{Lejay}
A.~Lejay.
\newblock {On the constructions of the skew Brownian motion}.
\newblock {\em Probability Surveys}, 3:413--466, 2006.

\bibitem{Li2006}
Z.-h. Li.
\newblock Branching processes with immigration and related topics.
\newblock {\em Front. Math. China}, 1(1):73--97, 2006.

\bibitem{lupu-diffusion-loops}
T.~Lupu.
\newblock Poisson ensembles of loops of one-dimensional diffusions.
\newblock {\em arXiv preprint arXiv:1302.3773}, 2013.

\bibitem{MR2250510}
M.~B. Marcus and J.~Rosen.
\newblock {\em Markov processes, {G}aussian processes, and local times}, volume
  100 of {\em Cambridge Studies in Advanced Mathematics}.
\newblock Cambridge University Press, Cambridge, 2006.

\bibitem{Pal13}
S.~Pal.
\newblock Wright-{F}isher diffusion with negative mutation rates.
\newblock {\em Ann. Probab.}, 41(2):503--526, 2013.

\bibitem{MR620995}
J.~Pitman and M.~Yor.
\newblock Bessel processes and infinitely divisible laws.
\newblock In {\em Stochastic integrals ({P}roc. {S}ympos., {U}niv. {D}urham,
  {D}urham, 1980)}, volume 851 of {\em Lecture Notes in Math.}, pages 285--370.
  Springer, Berlin, 1981.

\bibitem{PitmYor82}
J.~Pitman and M.~Yor.
\newblock A decomposition of {B}essel bridges.
\newblock {\em Z. Wahrsch. Verw. Gebiete}, 59(4):425--457, 1982.

\bibitem{RevuzYor}
D.~Revuz and M.~Yor.
\newblock {\em Continuous martingales and {B}rownian motion}, volume 293 of
  {\em Grundlehren der Mathematischen Wissenschaften [Fundamental Principles of
  Mathematical Sciences]}.
\newblock Springer-Verlag, Berlin, third edition, 1999.

\bibitem{RogWil2}
L.~C.~G. Rogers and D.~Williams.
\newblock {\em Diffusions, {M}arkov processes, and martingales. {V}ol. 2}.
\newblock Cambridge Mathematical Library. Cambridge University Press,
  Cambridge, 2000.
\newblock It\^o calculus, Reprint of the second (1994) edition.

\bibitem{ShiWat73}
T.~Shiga and S.~Watanabe.
\newblock Bessel diffusions as a one-parameter family of diffusion processes.
\newblock {\em Z. Wahrscheinlichkeitstheorie und Verw. Gebiete}, 27:37--46,
  1973.

\bibitem{Vatutin1977}
V.~A. Vatutin.
\newblock A critical {G}alton-{W}atson branching process with immigration.
\newblock {\em Teor. Verojatnost. i Primenen.}, 22(3):482--497, 1977.

\bibitem{VatZub1993}
V.~A. Vatutin and A.~M. Zubkov.
\newblock Branching processes. {II}.
\newblock {\em J. Soviet Math.}, 67(6):3407--3485, 1993.
\newblock Probability theory and mathematical statistics, 1.

\bibitem{Warren1997}
J.~Warren.
\newblock Branching processes, the {R}ay-{K}night theorem, and sticky
  {B}rownian motion.
\newblock In {\em S\'eminaire de {P}robabilit\'es, {XXXI}}, volume 1655 of {\em
  Lecture Notes in Math.}, pages 1--15. Springer, Berlin, 1997.

\bibitem{Winkel}
M.~Winkel.
\newblock Right inverses of nonsymmetric {L}\'evy processes.
\newblock {\em Ann. Probab.}, 30(1):382--415, 2002.

\bibitem{YenYor}
J.-Y. Yen and M.~Yor.
\newblock {\em Local times and excursion theory for {B}rownian motion}, volume
  2088 of {\em Lecture Notes in Mathematics}.
\newblock Springer, Cham, 2013.
\newblock A tale of Wiener and It\^o measures.

\bibitem{YorAspects1}
M.~Yor.
\newblock {\em {Some aspects of Brownian motion. Part I: Some special
  functionals}}.
\newblock Lectures in Math., ETH Z\"urich, Birkh\"auser, 1992.

\end{thebibliography}

\end{document}